\documentclass[11pt]{amsart}

\usepackage{amsmath}
\usepackage{amssymb}
\usepackage{latexsym}
\usepackage{amscd}
\usepackage{citesort}
\usepackage{mathrsfs}
\usepackage[all]{xy}

\newdimen\AAdi%
\newbox\AAbo%
\def\AAk#1#2{\s_etbox\AAbo=\hbox{#2}\AAdi=\wd\AAbo\kern#1\AAdi{}}%
\def\AAr#1#2#3{\s_etbox\AAbo=\hbox{#2}\AAdi=\ht\AAbo\raise#1\AAdi\hbox{#3}}%
\font\tenmsb=msbm10 at 12pt \font\sevenmsb=msbm7 at 8pt
\font\fivemsb=msbm5 at 6pt
\newfam\msbfam
\textfont\msbfam=\tenmsb \scriptfont\msbfam=\sevenmsb
\scriptscriptfont\msbfam=\fivemsb

\newtheorem{theorem}{Theorem}

\newtheorem{lemma}[theorem]{Lemma}

\numberwithin{equation}{section} \numberwithin{theorem}{section}

\renewcommand{\topmargin}{0cm}
\renewcommand{\oddsidemargin}{5mm}
\renewcommand{\evensidemargin}{5mm}
\renewcommand{\textwidth}{150mm}
\renewcommand{\textheight}{230mm}

\def\R{\mathbb R}

\def\N{\mathbb N}
\def\S{\mathbb S}

\def\na{\nabla}

\def\f#1#2{\frac{#1}{#2}}

\def\a{\alpha}
\def\be{\beta}

\def\r{\Re_{I\!V}}

\def\p#1{\partial #1}

\def\de{\delta}
\def\De{\Delta}

\def\ep{\epsilon}

\def\G{\Gamma}
\def\g{\gamma}
\def\k{\kappa}
\def\la{\lambda}

\def\lan{\langle}
\def\ran{\rangle}

\def\Om{\Omega}

\def\Si{\Sigma}

\def\r{\rho}

\begin{document}

\title[A rigidity theorem on the second fundamental form for self-shrinkers]
{A rigidity theorem on the second fundamental form for self-shrinkers}

\author{Qi Ding}
\address{Shanghai Center for Mathematical Sciences, Fudan University, Shanghai 200433, China}
\email{dingqi@fudan.edu.cn}\email{dingqi09@fudan.edu.cn}

\thanks{The author would like to thank Yuanlong Xin for his interest in this work. He is supported partially by NSFC}

\date{}
\begin{abstract}
In Theorem 3.1 of \cite{DX2}, we proved a rigidity result for self-shrinkers under the integral condition on the norm of the second fundamental form.
In this paper, we relax the such bound to any finite constant (see Theorem \ref{main} for details).
\end{abstract}

\maketitle

\section{Introduction}

Self-similar solutions for mean curvature flow play a key role in the understanding the possible singularities that the flow goes through. Self-shrinkers are type I singularity models of the flow. Huisken made a pioneer work on self-shrinking solutions of the flow \cite{H1,H2}.
Colding and Minicozzi \cite{CM1} gave a comprehensive study for self-shrinking hypersurfaces and solve a long-standing conjecture raised by Huisken.

Colding-Ilmanen-Minicozzi \cite{CIM} showed that cylindrical self-shrinkers are rigid in a very strong sense. Namely, any other shrinker that is sufficiently close to one of them on a large, but compact set must itself be a round cylinder. See \cite{GZ} by Guang-Zhu for further results.
Lu Wang in \cite{W1,W2} proved strong uniqueness theorems for self-shrinkers
asymptotic to regular cones or generalized cylinders of infinite order.

For Bernstein type theorems, Ecker-Huisken \cite{EH89} and Wang \cite{W0} showed the nonexistence of nontrivial graphic self-shrinking hypersurfaces in Euclidean space. For $2\le n\le6$, Guang-Zhu showed that any smooth complete self-shrinker in $\R^{n+1}$ which is graphical inside a large, but compact, set must be a hyperplane. Ding-Xin-Yang \cite{DXY} studied the sharp rigidity theorems with the condition on Gauss map of self-shrinkers. In high codimensions, see \cite{CCY,CH,DW,DX4,HW} for more Bernstein type theorems.

Le-Sesum \cite{LS1} showed that any complete embedded self-shrinking hypersurface with polynomial volume growth must be a hyperplane provided the squared norm of the second fundamental form $|B|^2<\f12$.  Cao-Li \cite{CL} showed that any complete self-shrinker (with high codimension) with polynomial volume growth must be a generalized cylinder provided $|B|^2\le\f12$. Later, Cheng-Peng \cite{CP} removed the condition of polynomial volume growth in the case of $|B|^2<\f12$ (See \cite{DX2,CO,CW,XX} for more results on the gap theorems of the norm of the second fundamental form). In \cite{DX2}, Ding-Xin proved a rigidity result for self-shrinkers if the integration of $|B|^n$ is small. In this paper, we improve the small constant to any finite constant.

For a complete properly immersed self-shrinker $\Si^n\subset\R^{n+1}$, Ilmanen showed that there exists a cone $\mathcal{C}\subset\R^{n+1}$ with the cross section being a compact set in $\S^n$ such that $\la\Si^n\rightarrow \mathcal{C}$ as $\la\rightarrow 0_+$ locally in the Hausdorff metric on closed sets (see \cite{I2} Lecture 2, B, remark on p.8).
In \cite{So}, Song gave a simple proof by "maximum principle for self-shrinkers".
For high codimensions, with backward heat kernel (see \cite{CM1}) we show the uniqueness of tangent cones at infinity for self-shrinkers with Euclidean volume growth in the current sense with the condition on mean curvature(see Theorem \ref{uinfcone}).


$\ep$-regularity theorems for the mean curvature flow have been studied by Ecker \cite{E0,E}, Han-Sun \cite{HS}, Ilmanen \cite{I1}, Le-Sesum \cite{LS}. Now we use the one showed by Ecker \cite{E} starting from self-similar solutions, and obtain the curvature estimates for self-shrinkers, see Theorem \ref{cess}.
Combining Theorem \ref{uinfcone}, Theorem \ref{cess} and backward uniqueness for parabolic operators \cite{ESS}, we can show that self-shrinkers with finite integration on $|B|^n$ must be planes, which improves a previous rigidity theorem in \cite{DX2}. A litter more, we obtain the following Theorem.
\begin{theorem}
Let $M$ be an $n$-dimensional properly non-compact self-shrinker with compact boundary in $\R^{n+m}$, $B$ denote the second fundamental form of $M$. If
\begin{equation}\aligned\label{Bn0}
\lim_{r\rightarrow\infty}\int_{M\cap{B_{2r}\setminus B_r}}|B|^nd\mu=0,
\endaligned
\end{equation}
then $M$ must be an $n$-plane through the origin.
\end{theorem}

\section{Preliminary}

Let $M$ be an $n$-dimensional $C^2$-submanifold in $\R^{n+m}$ with the induced metric. Let $\na$ and $\overline{\na}$ be the Levi-Civita
connections on $M$ and $\R^{n+m}$, respectively.  We define the second fundamental form $B$ of $M$ by $$B(V,W)=(\overline{\na}_VW)^N=\overline{\na}_VW-\na_VW$$
for any $V,W\in\G(TM)$, where the mean curvature vector $H$ of $M$ is given by $H=\mathrm{trace}(B)=\sum_{i=1}^nB(e_i,e_i),$ where $\{e_i\}$ is a local orthonormal frame field of $M$.

In this paper, $M^n$ is said to be a \emph{self-shrinker} in $\R^{n+m}$ if its mean curvature vector satisfies
\begin{equation}\label{0.1}
H= -\frac{X^N}{2},
\end{equation}
where $X=(x_1,\cdots,x_{n+m})\in\R^{n+m}$ is the position vector of $M$ in $\R^{n+m}$, and $(\cdots)^N$ stands for the orthogonal projection into the normal bundle $NM$.
Let $(\cdots)^T$ denote the orthogonal projection into the tangent bundle $TM$.

We define a second order differential operator $\mathcal{L}$ as in \cite{CM1} by
$$\mathcal{L}f=e^{\f{|X|^2}4}\mathrm{div}\left( e^{-\f{|X|^2}4}\na f\right)=\De f-\f12\lan X,\na f\ran$$
for any $f\in C^2(M)$. Let $\De$ be the Laplacian of $M$, then for self-shrinkers,
\begin{equation}\aligned\label{laplace}
\De|X|^2=2\lan X,\De X\ran+2|\na X|^2=2\lan X,H\ran+2n=-|X^N|^2+2n.
\endaligned
\end{equation}

In \cite{CM1}, Colding and Minicozzi defined a function $F_{X_0,t_0}$ for self-shrinking hypersurfaces in Euclidean space. Obviously, hypersurfaces can be generalized to submanifolds naturally in this definition. Set $\Phi_t\in C^\infty(\R^{n+m})$ for any $t>0$ by
$$\Phi_t(X)=\f1{(4\pi t)^{n/2}}e^{-\f{|X|^2}{4t}}.$$
For an $n$-complete submanifold $M$ in $\R^{n+m}$, we define a functional $F_t$ on $M$ by
$$F_t(M)=\int_M\Phi_td\mu=\f1{(4\pi t)^{n/2}}\int_{M}e^{-\f{|X|^2}{4t}}d\mu \quad \mathrm{for} \quad t>0,$$
where $d\mu$ is the volume element of $M$.
Sometimes, we write $F_t$ for simplicity if no ambiguous in the text.
If a self-shrinker is proper, then it is equivalent to that it has Euclidean volume growth at most by \cite{CZ} and \cite{DX}.
We shall only consider proper self-shrinkers in the following text.


Now we use the backward heat kernel to give a monotonicity formula for self-shrinkers with arbitrary codimensions, which is essentially same as self-shrinking hypersurfaces established by Colding-Minicozzi in \cite{CM1}.

\begin{lemma}\label{Vol}
For any $0<t_1\le t_2\le\infty$, each complete immersed self-shrinker $M^n$ with boundary $\p M$ (may be empty) in
$\R^{n+m}$ satisfies
\begin{equation}\aligned
F_{t_2}(M)-F_{t_1}(M)=&-\int_{t_1}^{t_2}\left(\int_{\p M}\lan X^T,\nu_{\p M}\ran\f{\Phi_s(X)}{2s}\right)ds\\
&+\int_{t_1}^{t_2} \f{1}{4s}\left(1-\f{1}s\right)\left(\int_{M}|X^N|^2\Phi_s(X)d\mu\right) ds.
\endaligned
\end{equation}
\end{lemma}
\begin{proof}
We differential $F_t(M)$ with respect to $t$,
\begin{equation}\aligned\label{Ft'=}
F_t'=(4\pi)^{-\f n2} t^{-(\f n2+1)}\int_{M} \left(-\f n2+\f{|X|^2}{4t}\right)e^{-\f{|X|^2}{4t}}d\mu.
\endaligned
\end{equation}
A straightforward calculation shows (see also \cite{DX})
\begin{equation}\aligned\label{1.4}
-e^{\f{|X|^2}{4t}}\mathrm{div}\left(e^{-\f{|X|^2}{4t}}\na|X|^2\right)=&-\De|X|^2+\f1{4t}\na|X|^2\cdot\na|X|^2\\
=&-2\langle H,X\rangle-2n+\f1{t}|X^T|^2\\
=&|X^N|^2+\f{|X^T|^2}t-2n\\
=&\left(1-\f{1}t\right)|X^N|^2+\f{|X|^2}t-2n,
\endaligned
\end{equation}
where the third equality above uses the self-shrinkers' equation (\ref{0.1}).
Then
\begin{equation}\aligned\label{1.2}
F_t'=&(4\pi)^{-\f n2}t^{-(\f n2+1)}\int_{M} \bigg(-\f14\mathrm{div}\left(e^{-\f{|X|^2}{4t}}\na|X|^2\right)-\f14\left(1-\f{1}t\right)|X^N|^2e^{-\f{|X|^2}{4t}}\bigg)d\mu\\
=&\f14(4\pi)^{-\f n2}t^{-(\f n2+1)}\left(-2\int_{\p M}\lan X^T,\nu_{\p M}\ran e^{-\f{|X|^2}{4t}}-\left(1-\f{1}t\right)\int_{M}|X^N|^2e^{-\f{|X|^2}{4t}}d\mu\right)\\
=&-\f1{2t}\int_{\p M}\lan X^T,\nu_{\p M}\ran \Phi_t(X)-\f1{4t}\left(1-\f{1}t\right)\int_{M}|X^N|^2\Phi_t(X)d\mu,
\endaligned
\end{equation}
where $\nu_{\p M}$ is the normal vector of $\p M$ in $\G(TM)$.
Then we complete the proof by integration from $t_1$ to $t_2$.
\end{proof}

Denote
\begin{equation}\aligned
G_t(M)\triangleq F'_t(M)+\f1{2t}\int_{\p M}\lan X^T,\nu_{\p M}\ran \Phi_t(X)=-\f1{4t}\left(1-\f{1}t\right)\int_{M}|X^N|^2\Phi_t(X)d\mu.
\endaligned
\end{equation}
The above Lemma implies $G_t(M)\le0$ for each self-shrinker and $t\ge1$. If $\p M$ is bounded and has finite $(n-1)$-dimensional Hausdorff measure, then the limit
$$\lim_{t\rightarrow\infty}\left(\int_1^tG_s(M)ds\right)$$
always exists, and is a finite negative number. Hence, it's clear that $\lim_{t\rightarrow\infty}F_t(M)$ exists.

\section{Uniqueness of tangent cones at infinity for self-shrinkers}

For any $n$-rectifiable varifold $V\subset\R^{n+m}$ with multiplicity one, we define a functional $\Xi_t$ by
$$\Xi_t(V,f)=\f1{(4\pi t)^{n/2}}\int_{\mathrm{spt} V}f e^{-\f{|X|^2}{4t}}d\mu_V$$
for any $t>0$, where $\mu_V$ is a measure on $\R^{n+m}$ associated with the Radon measure of $V$ in $\R^{n+m}\times G(n,n+m)$.

We suppose that $M$ is a self-shrinker in $\R^{n+m}\setminus B_R$ with boundary $\p M\subset\p B_R$ for some $R\ge1$ and $\mathcal{H}^{n-1}(\p M)<\infty$.
Let $\phi\in C^1(\R^{n+m}\setminus\{0\})$ be a homogeneous function of degree zero. Namely, for any $0\neq X\in\R^{n+m}$,
$$\phi(X)=\phi(|X|\xi)=\phi(\xi)$$
with $\xi=\f{X}{|X|}$. Then
\begin{equation}\aligned
\p_{x_i}\phi=\sum_j\left(\f{\de_{ij}}{|X|}-\f{x_ix_j}{|X|^3}\right)\p_{\xi_j}\phi,
\endaligned
\end{equation}
and
\begin{equation}\aligned\label{naphi2}
|\overline{\na}\phi|^2=\sum_{j,k}\left(\f{\de_{jk}}{|X|^2}-\f{x_jx_k}{|X|^4}\right)\p_{\xi_j}\phi\p_{\xi_k}\phi\le|X|^{-2}\sum_{j}\left(\p_{\xi_j}\phi\right)^2
\triangleq|X|^{-2}|\phi|_1^2.
\endaligned
\end{equation}
Taking the derivative of $\Xi_t(M,\phi)$ on $t$ gets
\begin{equation}\aligned\label{4.3}
&\p_t\Xi_t(M,\phi)=(4\pi)^{-\f n2} t^{-(\f n2+1)}\int_{M} \left(-\f n2+\f{|X|^2}{4t}\right)\phi e^{-\f{|X|^2}{4t}}d\mu\\
=&(4\pi)^{-\f n2}t^{-(\f n2+1)}\int_{M} \left(-\f{\phi}4\mathrm{div}\left(e^{-\f{|X|^2}{4t}}\na|X|^2\right)-\f{\phi} 4\left(1-\f{1}t\right)|X^N|^2e^{-\f{|X|^2}{4t}}\right)d\mu.
\endaligned
\end{equation}
Combining $X\cdot\overline{\na}\phi=0$, we have
\begin{equation}\aligned\label{phidive|X|}
&\int_{M} -\f{\phi}4\mathrm{div}\left(e^{-\f{|X|^2}{4t}}\na|X|^2\right)d\mu\\
=&\int_{M} -\f{1}4\mathrm{div}\left(\phi e^{-\f{|X|^2}{4t}}\na|X|^2\right)d\mu+\int_{M} \f{1}4\na\phi\cdot \na|X|^2 e^{-\f{|X|^2}{4t}}d\mu\\
=&-\f12\int_{\p M}\phi\lan X^T,\nu_{\p M}\ran e^{-\f{|X|^2}{4t}}+\int_{M} \f{1}2X\cdot\na\phi e^{-\f{|X|^2}{4t}}d\mu\\
=&-\f12\int_{\p M}\phi\lan X^T,\nu_{\p M}\ran e^{-\f{R^2}{4t}}-\f12\int_MX^N\cdot\overline{\na}\phi e^{-\f{|X|^2}{4t}}d\mu.
\endaligned
\end{equation}

Set $c_R=2^{-1}(4\pi)^{-\f n2}R\cdot\mathcal{H}^{n-1}(\p M)$. Substituting \eqref{naphi2} and \eqref{phidive|X|} into \eqref{4.3} gets
\begin{equation}\aligned
&|\p_t\Xi_t(M,\phi)|\le2^{-1}(4\pi)^{-\f n2}t^{-(\f n2+1)}\bigg(\int_{M}|X^N|\cdot|\overline{\na}\phi|e^{-\f{|X|^2}{4t}}d\mu\\
&\qquad\qquad\qquad\qquad\qquad\qquad\qquad\qquad \ +|\phi|_0Re^{-\f{R^2}{4t}}\mathcal{H}^{n-1}(\p M)\bigg)+|\phi|_0|G_t(M)|\\
\le&2^{-1}(4\pi)^{-\f n2}t^{-(\f n2+1)}\int_{M}\f{|X^N|}{|X|}|\phi|_1e^{-\f{|X|^2}{4t}}d\mu+|\phi|_0\left(|G_t(M)|+c_Rt^{-(\f n2+1)}\right)\\
\le&|\phi|_0\left(|G_t(M)|+c_Rt^{-(\f n2+1)}\right)\\
&+2^{-1}(4\pi)^{-\f n2}t^{-(\f n2+1)}|\phi|_1\left(\int_{M}|X^N|^2e^{-\f{|X|^2}{4t}}d\mu\right)^{1/2}\left(\int_{M}|X|^{-2}e^{-\f{|X|^2}{4t}}d\mu\right)^{1/2}\\
\le&|\phi|_0\left(|G_t(M)|+c_Rt^{-(\f n2+1)}\right)\\
&+|\phi|_1\left|G_t(M)\right|^{1/2}\sqrt{\f{t}{t-1}}\left((4\pi)^{-\f n2}t^{-(\f n2+2)}\int_{M}|X|^{-2}e^{-\f{|X|^2}{4t}}d\mu\right)^{1/2}.
\endaligned
\end{equation}
Put $D_r=M\cap B_r$ for every $r>0$. There is a constant $c_0>0$ depending only on $M$ such that for all $r>0$
$$\int_{D_r}1d\mu<c_0r^n.$$
Note $M\subset\R^{n+m}\setminus B_R$.
Then for $n\ge2$, $t\ge R^2$, one has
\begin{equation}\aligned
t^{-\f n2}\int_{M}\f{t}{|X|^2}e^{-\f{|X|^2}{4t}}d\mu\le& t^{-\f n2}\sum_{k=-1-[\f{\log (tR^{-2})}{2\log2}]}^{\infty}\int_{D_{2^{k+1}\sqrt{t}}\setminus D_{2^{k}\sqrt{t}}}\f{t}{|X|^2}e^{-\f{|X|^2}{4t}}d\mu\\
\le&t^{-\f n2}\sum_{k=-1-[\f{\log (tR^{-2})}{2\log2}]}^{\infty}\f{1}{4^{k}}e^{-4^{k-1}}\int_{D_{2^{k+1}\sqrt{t}}\setminus D_{2^{k}\sqrt{t}}}1d\mu\\
\le&c_0\sum_{k=0}^{\infty}4^{-k}e^{-4^{k}}2^{(k+1)n}+c_0\sum_{k=-1-[\f{\log (tR^{-2})}{2\log2}]}^{-1}4^{-k}2^{(k+1)n}\\
\le&c_0\sum_{k=0}^{\infty}2^{k(n-2)+n}e^{-4^{k-1}}+c_0\sum_{k=1}^{1+[\f{\log (tR^{-2})}{2\log2}]}2^{-k(n-2)+n}\\
\le& (4\pi)^{\f n2}c_1\left(1+\log t-2\log R\right),
\endaligned
\end{equation}
where $c_1$ is a constant depending only on $n,c_0$. Therefore
\begin{equation}\aligned\label{limphi}
|\p_t\Xi_t(M,\phi)|\le&\sqrt{c_1}\f{\sqrt{1+\log t}}t|\phi|_1\left|\f{t}{t-1}G_t(M)\right|^{1/2}+|\phi|_0\left(|G_t(M)|+c_Rt^{-(\f n2+1)}\right)\\
\le&c_1\f{1+\log t}{4t(t-1)}|\phi|_1+c_Rt^{-(\f n2+1)}|\phi|_0+\left(|\phi|_0+|\phi|_1\right)|G_t(M)|.
\endaligned
\end{equation}

\begin{theorem}\label{SIcone}
Let $M$ be an $n$-dimensional self-shrinker in $\R^{n+m}$ with Euclidean volume growth and boundary $\p M\subset\p B_R$. If
\begin{equation}\aligned\label{lim|H|}
\limsup_{r\rightarrow\infty}\left(r^{1-n}\int_{M\cap B_r}|H|\right)<\infty,
\endaligned
\end{equation}
then there is a sequence $t_i\rightarrow\infty$ such that $$M_{t_i}\triangleq t_i^{-1}M=\{X\in\R^{n+m}|\ t_iX\in M\}$$
converges to a cone $C$ in $\R^{n+m}$.
\end{theorem}
\begin{proof}
By co-area formula,
we can choose $R'>0$ so that $\mathcal{H}^{n-1}(\p M)<\infty$ with $\p M\subset\p B_{R'}$. Denote $R'$ by $R$ for convenience. Let $M_{t}\triangleq t^{-1}M=\{X\in\R^{n+m}|\ tX\in M\}$ for any $t>0$.
Since $M$ has Euclidean volume growth and \eqref{lim|H|} holds, then by compactness of varifolds, there exists an $n$-rectifiable varifold $T$ in $\R^{n+m}$ with integer multiplicity and a sequence of $t_i$ such that $M_{t_i}=t_i^{-1}M\rightharpoonup T$ in the sense of Radon measure (See 42.7 Theorem of \cite{S} for example).

Denote $\phi$ and $\Xi_{t}(M,\phi)$ as above. Set $\mu_t$ be the volume element of $M_t$. Since
\begin{equation}\aligned
\Xi_{t^2}(M,\phi)=\f1{(4\pi t^2)^{n/2}}\int_{M}\phi e^{-\f{|X|^2}{4t^2}}d\mu=\f1{(4\pi)^{n/2}}\int_{M_t}\phi e^{-\f{|X|^2}{4}}d\mu_{t}=\Xi_1(M_t,\phi),
\endaligned
\end{equation}
then for all $R>0$
\begin{equation}\aligned
\lim_{i\rightarrow\infty}\Xi_1(M_{t_iR},\phi)=\lim_{i\rightarrow\infty}\Xi_{R^2}(M_{t_i},\phi)=\f1{(4\pi R^2)^{n/2}}\int_{T}\phi\ e^{-\f{|X|^2}{4R^2}}d\mu_{T}=\Xi_{R^2}(T,\phi).
\endaligned
\end{equation}
Note that $G_t(M)$ does not change sign for $t>1$.
Fixing $0<r<R<\infty$, from \eqref{limphi} we have
\begin{equation}\aligned
&\left|\Xi_{t_i^2r^2}(M,\phi)-\Xi_{t_i^2R^2}(M,\phi)\right|\le\int_{{t_i^2r^2}}^{{t_i^2R^2}}|\p_s\Xi_s(M,\phi)|ds\\
\le&\int_{{t_i^2r^2}}^{{t_i^2R^2}}\left(c_1\f{1+\log s}{4s(s-1)}|\phi|_1+c_R|\phi|_0s^{-(\f n2+1)}+\left(|\phi|_0+|\phi|_1\right)|G_s(M)|\right)ds\\
\le&\f{c_1}4|\phi|_1\int_{t_i^2r^2}^{\infty}\f{1+\log s}{s(s-1)}ds+\f{2}n\left(t_ir\right)^{-n-2}c_R|\phi|_0+\left(|\phi|_0+|\phi|_1\right)\left|\int_{{t_i^2r^2}}^{{t_i^2R^2}}G_s(M)ds\right|
\endaligned
\end{equation}
for all $t_i$ with $rt_i\ge2$. Since
\begin{equation}\aligned
&\left|\int_{{t_i^2r^2}}^{{t_i^2R^2}}G_s(M)ds\right|\le\left|\int_{{t_i^2r^2}}^{{t_i^2R^2}}F'_t(M)ds+\int_{{t_i^2r^2}}^{{t_i^2R^2}}\left(\f1{2s}\int_{\p M}\lan X^T,\nu_{\p M}\ran \Phi_s(X)\right)ds\right|\\
\le&\left|F_{t_i^2r^2}(M)-F_{t_i^2R^2}(M)\right|+\int_{{t_i^2r^2}}^{{t_i^2R^2}}\left(\f R{2s}\mathcal{H}^{n-1}(\p M)(4\pi s)^{-n/2} \right)ds\\
=&\left|F_{t_i^2r^2}(M)-F_{t_i^2R^2}(M)\right|+\f Rn(4\pi)^{-n/2} \mathcal{H}^{n-1}(\p M)(t_ir)^{-n}
\endaligned
\end{equation}
and $\lim_{t\rightarrow\infty}F_t$ exists, we obtain
\begin{equation}\aligned
\lim_{i\rightarrow\infty}\Xi_1(M_{t_ir},\phi)=\lim_{i\rightarrow\infty}\Xi_1(M_{t_iR},\phi)=\Xi_{R^2}(T,\phi).
\endaligned
\end{equation}
Hence
\begin{equation}\aligned
\Xi_{t}(T,\phi)=\f1{(4\pi t)^{n/2}}\int_{T}\phi e^{-\f{|X|^2}{4t}}d\mu_T
\endaligned
\end{equation}
is independent of $t\in(0,\infty)$.

Clearly,
$$0<\mathcal{H}^n(T\cap B_r)\le c_2r^n$$
for some constant $c_2>0$ and all $r>0$.
By the following lemma for $V(r)=\int_{T\cap B_r}\phi\ d\mu_T$, we conclude that
\begin{equation}\aligned\label{rnCr}
r^{-n}\int_{T\cap B_r}\phi\ d\mu_T
\endaligned
\end{equation}
is a constant independent of $r$. 
An analog argument as the proof of 19.3 in \cite{S} implies that $T$ is a cone.
\end{proof}

\begin{lemma}
Let $V(r)$ be a monotone nondecreasing continuous function on $[0,\infty)$ with $V(0)=0$ and $V(r)\le c_3r^n$ for some constant $c_3>0$. If the quantity
\begin{equation}\aligned
\f1{(4\pi t)^{n/2}}\int_0^{\infty} e^{-\f{r^2}{4t}}dV(r)
\endaligned
\end{equation}
is a constant for any $t>0$, then $r^{-n}V(r)$ is a constant.
\end{lemma}
\begin{proof}
There are constants $\k_0,\k_1>0$ such that for all $t>0$
\begin{equation}\aligned
\int_0^{\infty} e^{-\f{r^2}{t}}dV(r)=\k_0 t^{n/2}=\k_1\int_0^{\infty} e^{-\f{r^2}{t}}dr^n,
\endaligned
\end{equation}
namely,
\begin{equation}\aligned
\int_0^{\infty} e^{-\f{r^2}{t}}d\left(V(r)-\k_1r^n\right)=0.
\endaligned
\end{equation}
Integrating by parts implies
\begin{equation}\aligned\label{Vk1}
\int_0^{\infty} \left(V(r)-\k_1r^n\right)re^{-\f{r^2}{t}}dr=0.
\endaligned
\end{equation}
Suppose that there is a constant $r_0>0$ such that $V(r_0)-\k_1r_0^n>0$ (Or else we complete the proof by \eqref{Vk1}). Then there is a $0<\de<\f{r_0}2$ and $\ep>0$ such that $V(r)-\k_1r^n\ge\ep$ for all $r\in(r_0-\de,r_0+\de)$.
Set $t_p=\f2pr_0^2$, then in $(0,\infty)$ the function
$$r^pe^{-\f{r^2}{t_p}}$$
attains its maximal value at $r=r_0$.

Now we claim
\begin{equation}\aligned\label{largeP}
\lim_{p\rightarrow\infty}\f{p^{\f12}e^{\f p2}}{r_0^{p+1}}\int_{r_0-\de}^{r_0+\de}r^pe^{-\f{r^2}{t_p}}dr=\int_{-\infty}^{\infty}e^{-t^2}dt=\sqrt{\pi}.
\endaligned
\end{equation}
In fact,
\begin{equation}\aligned
I(p)\triangleq&\f{p^{\f12}e^{\f p2}}{r_0^{p+1}}\int_{r_0-\de}^{r_0+\de}r^pe^{-\f{r^2}{t_p}}dr
=p^{\f12}e^{\f p2}\int_{-\f{\de}{r_0}}^{\f{\de}{r_0}}(1+s)^pe^{-\f p2(1+s)^2}ds\\
=&\int_{-\f{\de}{r_0}\sqrt{p}}^{\f{\de}{r_0}\sqrt{p}}\left(1+\f t{\sqrt{p}}\right)^pe^{-\f p2\left(\f{2t}{\sqrt{p}}+\f{t^2}p\right)}dt\\
=&\int_{-\f{\de}{r_0}\sqrt{p}}^{\f{\de}{r_0}\sqrt{p}}e^{p\log\left(1+\f t{\sqrt{p}}\right)}e^{-\sqrt{p}t-\f{t^2}2}dt.
\endaligned
\end{equation}
When $-\f12\le s<\infty$, a simple calculation implies
$$\min\left\{0,\f83s^3\right\}\le\log(1+s)-s+\f{s^2}2\le\f{s^3}3.$$
Combining the above inequality, we get
\begin{equation}\aligned
&\limsup_{p\rightarrow\infty}I(p)\le\limsup_{p\rightarrow\infty}\int_{-\f{\de}{r_0}\sqrt{p}}^{\f{\de}{r_0}\sqrt{p}}
e^{-t^2+\f{t^3}{3\sqrt{p}}}dt\\
=&\lim_{p\rightarrow\infty}\int_{-\f{\de}{r_0}\sqrt{p}}^{\f{\de}{r_0}\sqrt{p}}
e^{-t^2(1-\f t{3\sqrt{p}})}dt=\int_{-\infty}^{\infty}e^{-t^2}dt,
\endaligned
\end{equation}
and
\begin{equation}\aligned
\liminf_{p\rightarrow\infty}I(p)\ge&\lim_{p\rightarrow\infty}\int_{0}^{\f{\de}{r_0}\sqrt{p}}
e^{-t^2}dt+\liminf_{p\rightarrow\infty}\int_{-\f{\de}{r_0}\sqrt{p}}^{0}
e^{-t^2+\f{8t^3}{3\sqrt{p}}}dt\\
=&\int_{0}^{\infty}e^{-t^2}dt+\lim_{p\rightarrow\infty}\int_{-\f{\de}{r_0}\sqrt{p}}^{0}
e^{-t^2\left(1-\f{8t}{3\sqrt{p}}\right)}dt=\int_{-\infty}^{\infty}e^{-t^2}dt.
\endaligned
\end{equation}
Hence we have shown \eqref{largeP}.

For $p>1$,
\begin{equation}\aligned
&\f{p^{\f12}e^{\f p2}}{r_0^{p+1}}\int^{\infty}_{r_0+\de}r^{n+p}e^{-\f{r^2}{t_p}}dr=r_0^n\int_{\f{\de}{r_0}\sqrt{p}}^{\infty}e^{(n+p)\log\left(1+\f t{\sqrt{p}}\right)}e^{-\sqrt{p}t-\f{t^2}2}dt\\
\le&r_0^n\int_{\f{\de}{r_0}\sqrt{p}}^{\infty}e^{(n+p)\f t{\sqrt{p}}}e^{-\sqrt{p}t-\f{t^2}2}dt\le r_0^n\int_{\f{\de}{r_0}\sqrt{p}}^{\infty}e^{\f n{\sqrt{p}}t-\f{t^2}2}dt.
\endaligned
\end{equation}
Then
\begin{equation}\aligned\label{pVrk1tp}
&\liminf_{p\rightarrow\infty}\f{p^{\f12}e^{\f p2}}{r_0^{p+1}}\int_{0}^{\infty}\left(V(r)-\k_1r^n\right)r^pe^{-\f{r^2}{t_p}}dr\\
\ge&\liminf_{p\rightarrow\infty}\f{p^{\f12}e^{\f p2}}{r_0^{p+1}}\left(\ep\int_{r_0-\de}^{r_0+\de}r^pe^{-\f{r^2}{t_p}}dr-\k_1\int_{0}^{r_0-\de}r^{n+p}e^{-\f{r^2}{t_p}}dr
-\k_1\int^{\infty}_{r_0+\de}r^{n+p}e^{-\f{r^2}{t_p}}dr\right)\\
\ge&\ep\sqrt{\pi}-\k_1r_0^n\limsup_{p\rightarrow\infty}\left(\f{p^{\f12}e^{\f p2}}{r_0^{p+1}}\int_{0}^{r_0-\de}r^{p}e^{-\f{pr^2}{2r_0^2}}dr+\int^{\infty}_{\f{\de}{r_0}\sqrt{p}}e^{\f n{\sqrt{p}}t-\f{t^2}2}dr\right)\\
=&\ep\sqrt{\pi}-\k_1r_0^n\limsup_{p\rightarrow\infty}\left(\int^{-\f{\de}{r_0}\sqrt{p}}_{-\sqrt{p}}e^{p\log\left(1+\f t{\sqrt{p}}\right)}e^{-\sqrt{p}t-\f{t^2}2}dt+\int^{\infty}_{\f{\de}{r_0}\sqrt{p}}e^{-t^2\left(\f12-\f n{\sqrt{p}t}\right)}dr\right)\\
\ge&\ep\sqrt{\pi}-\k_1r_0^n\limsup_{p\rightarrow\infty}\left(\int^{-\f{\de}{r_0}\sqrt{p}}_{-\sqrt{p}}e^{\sqrt{p}t}e^{-\sqrt{p}t-\f{t^2}2}dt\right)=\ep\sqrt{\pi}.
\endaligned
\end{equation}
Taking the derivative of $t$ in \eqref{Vk1} yields
\begin{equation}\aligned
\int_{0}^{\infty}\left(V(r)-\k_1r^n\right)r^{2k+1}e^{-\f{r^2}{t}}dr=0
\endaligned
\end{equation}
for any $t>0$ and $k=0,1,2\cdots$. If we choose $p=2k+1$, $r_0^2>e$, $t_p=\f2pr_0^2$ in \eqref{pVrk1tp}, then we get the contradiction provided $k$ is sufficiently large. Hence $V(r)-\k_1r^n\equiv0$.
\end{proof}

\begin{theorem}\label{uinfcone}
Let $M$ be an $n$ dimensional smooth self-shrinker with Euclidean volume growth and boundary $\p M\subset\p B_R$ in $\R^{n+m}$. If \eqref{lim|H|} holds, then the limit
$\lim_{r\rightarrow\infty}r^{-1}M$ exists and is cone, namely, the tangent cone at infinity of $M$ is a unique cone.
\end{theorem}
\begin{proof}
We claim
\begin{equation}\aligned\label{Drphi}
\lim_{r\rightarrow\infty}\left(r^{-n}\int_{M\cap B_r}\phi d\mu\right)
\endaligned
\end{equation}
exists for every homogeneous function $\phi\in C^1(\R^{n+m}\setminus\{0\})$ with degree zero.
Suppose
\begin{equation}\aligned
\limsup_{r\rightarrow\infty}r^{-n}\int_{M\cap B_r}\phi d\mu>\liminf_{r\rightarrow\infty}r^{-n}\int_{M\cap B_r}\phi d\mu
\endaligned
\end{equation}
for some homogeneous function $\phi\in C^1(\R^{n+m}\setminus\{0\})$ with degree zero. Then there exist two sequences $p_i\rightarrow\infty$ and $q_i\rightarrow\infty$ such that
\begin{equation}\aligned
\lim_{i\rightarrow\infty}p_i^{-n}\int_{M\cap B_{p_i}}\phi d\mu>\lim_{i\rightarrow\infty}q_i^{-n}\int_{M\cap B_{q_i}}\phi d\mu.
\endaligned
\end{equation}
By compactness of varifolds and Theorem \ref{SIcone}, there exist two cones $C_1,C_2$ in $\R^{n+m}$ with integer multiplicities and subsequences $p_{k_i}$ of $p_i$ and $q_{k_i}$ of $q_i$ such that $M_{p_{k_i}}\rightharpoonup C_1$ and $M_{q_{k_i}}\rightharpoonup C_2$ in the sense of Radon measure. So we have
\begin{equation}\aligned
\int_{C_1\cap B_1}\phi d\mu_{C_1}=&\lim_{i\rightarrow\infty}\int_{M_{p_{k_i}}\cap B_1}\phi d\mu_{p_{k_i}}=
\lim_{i\rightarrow\infty}p_{k_i}^{-n}\int_{M\cap B_{p_{k_i}}}\phi d\mu\\
>&\lim_{i\rightarrow\infty}q_{k_i}^{-n}\int_{M\cap B_{q_{k_i}}}\phi d\mu=\lim_{i\rightarrow\infty}\int_{M_{q_{k_i}}\cap B_1}\phi d\mu_{q_{k_i}}\\
=&\int_{C_2\cap B_1}\phi d\mu_{C_2},
\endaligned
\end{equation}
which implies
\begin{equation}\aligned\label{Cpphi}
\int_{C_1}\phi e^{-\f{|X|^2}4} d\mu_{C_1}>\int_{C_2}\phi e^{-\f{|X|^2}4} d\mu_{C_2}
\endaligned
\end{equation}
by co-area formula.

From the previous argument, the limit
\begin{equation}\aligned
\lim_{t\rightarrow\infty}\Xi_{t}(M,\phi)=\lim_{t\rightarrow\infty}\f1{(4\pi t)^{n/2}}\int_{M}\phi e^{-\f{|X|^2}{4t}}d\mu
\endaligned
\end{equation}
exists. It infers that
\begin{equation}\aligned\label{CpCq}
\int_{C_1}\phi e^{-\f{|X|^2}4} d\mu_{C_1}=&\lim_{i\rightarrow\infty}\int_{M_{p_{k_i}}}\phi e^{-\f{|X|^2}4}=\lim_{t\rightarrow\infty}\f1{ t^{n/2}}\int_{M}\phi e^{-\f{|X|^2}{4t}}d\mu\\
=&\lim_{i\rightarrow\infty}\int_{M_{q_{k_i}}}\phi e^{-\f{|X|^2}4}=\int_{C_2}\phi e^{-\f{|X|^2}4} d\mu_{C_2}.
\endaligned
\end{equation}
However, \eqref{CpCq} contradicts \eqref{Cpphi}. Hence, the claim \eqref{Drphi} holds.

If $\lim_{i\rightarrow\infty}r_i^{-1}M\rightharpoonup C^+$, $\lim_{i\rightarrow\infty}s_i^{-1}M\rightharpoonup C^-$ and $C^+\neq C^-$ are cones, then from \eqref{CpCq} one has
\begin{equation}\aligned
\int_{C^+}\phi e^{-\f{|X|^2}4} d\mu_{C^+}=\int_{C^-}\phi e^{-\f{|X|^2}4} d\mu_{C^-}
\endaligned
\end{equation}
for every homogeneous function $\phi\in C^1(\R^{n+m}\setminus\{0\})$ with degree zero. It's clear that
\begin{equation}\aligned
\int_{C^+\cap \p B_1}\phi=\int_{C^-\cap \p B_1}\phi.
\endaligned
\end{equation}
Arbitrariness of $\phi$ implies $C^+=C^-$. Therefore, the tangent cone at infinity of $M$ is a unique cone.
\end{proof}


\section{A rigidity theorem for self-shrinkers}


Let us recall an $\ep$-regularity theorem for mean curvature flow showed by Ecker (A litter different from Theorem 1.8 in \cite{E}).
\begin{theorem}\label{epreg}
For $p\in[n,n+2]$, there exists a constant $\ep_0>0$ such that for any smooth properly immersed solution $\mathcal{M}=(\mathcal{M}_t)_{t\in(-4,0)}$ of mean curvature flow in $\R^{n+m}$, every $X_0$ which the solution reaches at time $t_0\in[-1,0)$, the assumption
\begin{equation}\aligned\label{DefIx0t0}
I_{X_0,t_0}\triangleq\sup_{\sqrt{-t_0}\le\r<\r'\le2}\f1{\left(\r'^2-\r^2\right)^{\f{n+2-p}2}}\int_{-\r'^2}^{-\r^2}\int_{\mathcal{M}_t\cap B_{2}(X_0)}|B|^p\le\ep_0
\endaligned
\end{equation}
implies
\begin{equation}\aligned
\sup_{\sigma\in[0,1]}\left(\sigma^2\sup_{t\in(t_0-(1-\sigma)^2,t_0)}\ \sup_{\mathcal{M}_t\cap B_{1-\sigma}(X_0)}|B|^2\right)\le\left(\ep_0^{-1}I_{X_0,t_0}\right)^{\f2p}.
\endaligned
\end{equation}
\end{theorem}
For completeness, we give the proof in appendix which is based on Ecker's proof.
Let us consider the mean curvature flow in Theorem \ref{epreg} which starts from a self-shrinker. Let $M$ be a self shrinker, then the one-parameter family $\mathcal{M}_t=\sqrt{-t}M$ is a mean curvature flow for $-4\le t<0$.
In this case,
\begin{equation}\aligned\label{r'rBp}
I_{X_0,t_0}=&\sup_{\sqrt{-t_0}\le\r<\r'\le2}\left(\r'^2-\r^2\right)^{-\f{n+2-p}2}\int_{-\r'^2}^{-\r^2}\left(\int_{\sqrt{-t}M\cap B_{2}(X_0)}|B|^p\right)dt\\
=&\sup_{\sqrt{-t_0}\le\r<\r'\le2}\left(\r'^2-\r^2\right)^{-\f{n+2-p}2}\int_{\f1{\r'}}^{\f1{\r}}\left(\int_{\f1rM\cap B_{2}(X_0)}|B|^p\right)\f2{r^3}dr\\
=&\sup_{\sqrt{-t_0}\le\r<\r'\le2}2\left(\r'^2-\r^2\right)^{-\f{n+2-p}2}\int_{\f1{\r'}}^{\f1{\r}}\left(r^{p-n-3}\int_{M\cap B_{2r}(rX_0)}|B|^pd\mu\right)dr.
\endaligned
\end{equation}
For any $-\f14<t_0<0$ and $X_0\in \sqrt{-t_0}M$, $I_{X_0,t_0}\le\ep_0$ implies
\begin{equation}\aligned
\f14\sup_{t\in(t_0-\f14,t_0)}\ \sup_{\sqrt{-t}M\cap B_{\f12}(X_0)}|B|^2\le\left(\ep_0^{-1}I_{X_0,t_0}\right)^{\f2p}.
\endaligned
\end{equation}
Hence
\begin{equation}\aligned\label{CESS}
\sup_{t\in(2,(-t_0)^{-1/2})}\left(\sup_{\f1tM\cap B_{\f12}(X_0)}|B|^2\right)\le 4\left(\ep_0^{-1}I_{X_0,t_0}\right)^{\f2p}.
\endaligned
\end{equation}

Now we have the following curvature estimates for self-shrinkers.
\begin{theorem}\label{cess}
Let $M$ be an $n$ dimensional proper self-shrinker in $\R^{n+m}$. If for some $p\in[n,n+2)$ there is
\begin{equation}\aligned\label{|B|p}
\lim_{R\rightarrow\infty}\int_{M\cap{B_{2R}\setminus B_R}}|B|^pd\mu=0,
\endaligned
\end{equation}
then there exist constants $c,r_0>0$ such that for all $r\ge r_0$ and $t>4$ we have
\begin{equation}\aligned\label{Brt}
\sup_{M\cap\p B_{(r+1)t}}|B|\le \f ct\left(\sup_{s\ge r}\int_{M\cap{B_{2s}\setminus B_s}}|B|^pd\mu\right)^{\f1p}.
\endaligned
\end{equation}
\end{theorem}
\begin{proof}
For any $\ep>0$, there exists a constant $r_0\ge2$ such that for any $r_1\ge r_0$ we have
$$\sup_{r\ge r_1}\int_{M\cap{B_{2r}\setminus B_r}}|B|^pd\mu<\ep.$$
For any vector $X_0\in\R^{n+m}$ with $|X_0|\ge 2r_1+2$, it's clear that 
$$B_{2r}(rX_0)\subset \left(B_{(|X_0|+2)r}\setminus B_{(|X_0|-2)r}\right)\subset \left(B_{2(|X_0|-2)r}\setminus B_{(|X_0|-2)r}\right).$$
Let $X\in\sqrt{-t}M$ with $|X|\ge 2r_1+2$ and $t<0$, then
\begin{equation}\aligned
\int_{M\cap B_{2r}(rX)}|B|^pd\mu\le\int_{M\cap \left(B_{2(|X|-2)r}\setminus B_{(|X|-2)r}\right)}|B|^pd\mu\le\sup_{s\ge r_1}\int_{M\cap{B_{2s}\setminus B_s}}|B|^pd\mu<\ep.
\endaligned
\end{equation}
In view of \eqref{r'rBp}, one has
\begin{equation}\aligned
&I_{X,t}\le\sup_{0\le\r<\r'\le2}\left(\r'^2-\r^2\right)^{-\f{n+2-p}2}\int_{\f1{\r'}}^{\f1{\r}}2r^{p-n-3}dr\cdot\sup_{r\ge r_1}\int_{M\cap{B_{2r}\setminus B_r}}|B|^pd\mu\\
\le&\f{2}{2+n-p}\sup_{0\le\r<\r'\le2}\left(\r'^2-\r^2\right)^{-\f{n+2-p}2}\left(\r'^{2+n-p}-\r^{2+n-p}\right)\sup_{r\ge r_1}\int_{M\cap{B_{2r}\setminus B_r}}|B|^pd\mu.
\endaligned
\end{equation}
Since for each fixed $\a\in(0,1]$ and each $s\ge1$,
\begin{equation}\aligned
\f{\p}{\p s}\left(\f{s^{2\a}-1}{\left(s^2-1\right)^{\a}}\right)=2\a\f{s-s^{2\a-1}}{\left(s^2-1\right)^{\a}}\ge0,
\endaligned
\end{equation}
then
$$\sup_{s\ge1}\f{s^{2\a}-1}{\left(s^2-1\right)^{\a}}=\lim_{s\rightarrow\infty}\f{s^{2\a}-1}{\left(s^2-1\right)^{\a}}=1.$$
So we obtain
\begin{equation}\aligned
I_{X,t}\le\f{2}{2+n-p}\sup_{r\ge r_1}\int_{M\cap{B_{2r}\setminus B_r}}|B|^pd\mu<\f{2\ep}{2+n-p}.
\endaligned
\end{equation}
Let $\ep=\f{2+n-p}2\ep_0$, $|X|\ge 2r_1+2$ and $-\f14<t<0$, then combining \eqref{CESS} we have
\begin{equation}\aligned
\sup_{s\in(2,(-t)^{-1/2})}\left(\sup_{\f1sM\cap B_{\f12}(X)}|B|\right)\le 2\left(\ep^{-1}\sup_{r\ge r_1}\int_{M\cap{B_{2r}\setminus B_r}}|B|^pd\mu\right)^{\f1p},
\endaligned
\end{equation}
which implies
\begin{equation}\aligned
2\left(\ep^{-1}\sup_{r\ge r_1}\int_{M\cap{B_{2r}\setminus B_r}}|B|^pd\mu\right)^{\f1p}\ge&\sup_{X\in\f1tM\cap\p B_{2r_1+2}}\left(\sup_{\f1tM\cap B_{\f12}(X)}|B|\right)\\
=&\sup_{|X|=2r_1+2,tX\in M}\left(t\sup_{M\cap B_{\f t2}(tX)}|B|\right)\\
\ge&t\sup_{M\cap\p B_{2t(r_1+1)}}|B|
\endaligned
\end{equation}
for any $r\ge r_1$ and $t>2$.
This suffices to complete the proof.
\end{proof}

\begin{lemma}\label{HpEVG}
Let $M$ be an $n$ dimensional proper noncompact self-shrinker in $\R^{n+m}$ with
\begin{equation}\aligned\label{Hpinfty}
\limsup_{r\rightarrow\infty}\int_{M\cap B_{2r}\setminus B_r}|H|^pd\mu<\infty
\endaligned
\end{equation}
for some $p\ge2$. Then every end of $M$ has Euclidean volume growth at least.
\end{lemma}
\begin{proof}
For any end $E$ of $M$, there is a constant $r_0>0$ such that $\p E\subset B_{r_0}$. Replacing $E$ by $E\setminus B_{r_0}$ if necessary, we have $\p E\subset\p B_{r_0}$. Set $E_r=E\cap B_r$.
For $0\le s<1$ and $r\ge r_0$, we have
\begin{equation}\aligned
&\f{\p}{\p r}\left(r^{-n+s}\int_{E_r}1d\mu\right)=-(n-s) r^{-n+s-1}\int_{E_r}1d\mu+r^{-n+s}\int_{E\cap\p B_r}\f{|X|}{|X^T|}\\
\ge&-(n-s) r^{-n+s-1}\int_{E_r}1d\mu+r^{-n+s-1}\int_{E\cap\p B_r}|X^T|\\
=&-(n-s) r^{-n+s-1}\int_{E_r}1d\mu+\f12r^{-n+s-1}\int_{E_r}\De|X|^2+r^{-n+s-1}\int_{\p E}|X^T|\\
\ge&s r^{-n+s-1}\int_{E_r}1d\mu-2 r^{-n+s-1}\int_{E_r}|H|^2d\mu\\
\ge&s r^{-n+s-1}\int_{E_r}1d\mu-2r^{-n+s-1}\left(\int_{E_r}|H|^p d\mu\right)^{\f2p}\left(\int_{E_r}1d\mu\right)^{1-\f2p}.
\endaligned
\end{equation}
Set
$$\widetilde{V}_s(r)=r^{-n+s}\int_{E_r}1d\mu,$$
then
\begin{equation}\aligned\label{prVs}
\p_r\widetilde{V}_s\ge&\f sr\widetilde{V}_s-2r^{-\f 2p(n-s)-1}\widetilde{V}_s^{1-\f2p}\left(\int_{E_r}|H|^p d\mu\right)^{\f2p}\\
=&\f {\widetilde{V}_s}r\left(s-2\left(\int_{E_r}|H|^p d\mu\right)^{\f2p}\left(\int_{E_r}1d\mu\right)^{-\f2p}\right).\\
\endaligned
\end{equation}

For any $r>0$, let $q\in\N$ with $2^q\le r<2^{q+1}$. By \eqref{Hpinfty}, there is a constant $c>0$ such that
\begin{equation}\aligned
\int_{E_r}|H|^pd\mu\le\sum_{k=0}^q\int_{E_{2^{k+1}}\setminus E_{2^k}}|H|^pd\mu+\int_{E_1}|H|^pd\mu\le c(q+2)\le c\left(\f{\log r}{\log 2}+2\right).
\endaligned
\end{equation}
From \cite{LW,MW}, every end of any self-shrinker has linear growth at least.
For any $\de>0$, there exists a constant $r_{\de}>0$ such that for all $r\ge r_{\de}$
$$\left(\int_{E_r}|H|^p d\mu\right)^{\f2p}\left(\int_{E_r}1d\mu\right)^{-\f2p}\le\f{\de}4,$$
then \eqref{prVs} implies
\begin{equation}\aligned
\p_r\widetilde{V}_{\de}\ge\f{\de\widetilde{V}_\de}{2r}.\\
\endaligned
\end{equation}
By Newton-Leibniz formula,
\begin{equation}\aligned\label{Vder}
\log\widetilde{V}_{\de}(r)\ge \log\widetilde{V}_{\de}(r_\de)+\int_{r_\de}^r\f{\p_s\widetilde{V}_{\de}(s)}{\widetilde{V}_{\de}(s)}ds\ge \log\widetilde{V}_{\de}(r_\de)+\f{\de}2\log\f r{r_\de}.
\endaligned
\end{equation}

Denote $\widetilde{V}(r)=\widetilde{V}_0(r)$. By \eqref{prVs},
\begin{equation}\aligned\label{prVf2pHp}
\p_r\widetilde{V}^{\f2p}\ge-\f4p\left(\int_{E_r}|H|^p d\mu\right)^{\f2p}\ r^{-\f{2n}p-1}.\\
\endaligned
\end{equation}
There is a constant $s_0>e$ such that for all $s\ge s_0$ the inequality $\log s<s^{\f np}$ holds.
Hence combining \eqref{Hpinfty} and \eqref{prVf2pHp}, for any $r_2\ge r_1\ge\max\{s_0,r_0\}$ we have
\begin{equation}\aligned\label{421V}
\widetilde{V}^{\f2p}(r_2)-\widetilde{V}^{\f2p}(r_1)\ge-\f {nc'} p\int_{r_1}^{r_2} r^{-\f{2n}p-1}\log rdr\ge-\f {nc'} p\int_{r_1}^{r_2} r^{-\f{n}p-1}dr\ge-c' r_1^{-\f{n}p}
\endaligned
\end{equation}
for some constant $c'>0$.
\eqref{Vder} infers
$$\lim_{r\rightarrow\infty}r^{\de}\widetilde{V}(r)=\infty$$
for any $\de>0$.
Combining \eqref{421V}, we obtain
\begin{equation}\aligned
\widetilde{V}^{\f2p}(r_2)\ge\f12\widetilde{V}^{\f2p}(r_1)>0
\endaligned
\end{equation}
for some fixed sufficiently large $r_1\ge\max\{s_0,r_0\}$.
This suffices to complete the proof.
\end{proof}



Now let us prove the following rigidity theorem.
\begin{theorem}\label{main}
Let $M$ be an $n$-dimensional properly non-compact self-shrinker with compact boundary in $\R^{n+m}$. If
\begin{equation}\aligned\label{Bn0}
\lim_{r\rightarrow\infty}\int_{M\cap{B_{2r}\setminus B_r}}|B|^nd\mu=0,
\endaligned
\end{equation}
then $M$ must be an $n$-plane through the origin.
\end{theorem}
\begin{proof}
From Theorem \ref{cess}, we obtain
\begin{equation}\aligned\label{r|B|r}
\lim_{r\rightarrow\infty}\left(r\sup_{B_{5r}}|B|\right)=0.
\endaligned
\end{equation}
Let $M_r=r^{-1}M$ for any $r>0$, then $M_t\cap\left(B_K\setminus B_{\f1K}\right)$ for any $K>0$ has bounded sectional curvature. On the one hand, $M_r\cap\left(B_K\setminus B_{\f1K}\right)$ converges to a smooth manifold with $C^{1,\a}$ metric in the Gromov-Hausdorff sense. On the other hand, Theorem \ref{uinfcone} implies that $M_r$ converges to a unique cone $C$ in $\R^{n+1}$ in the current sense. Hence for any $x\in C\setminus\{0\}$, there is a neighborhood $\Om_x$ of $x$ such that $\Om_x\cap C$ can be represented as a graph with $C^{1,\a}$ graphic function. Hence by Fatou lemma, $\Om_x\cap C$ is flat by \eqref{r|B|r}. So we conclude that $M_r$ converges to a union of finite $n$-planes through origin as $r\rightarrow\infty$.
Note that every end of $M$ converges to a union of finite $n$-planes through origin by Lemma \ref{HpEVG}.
Therefore, up to rotation there are a constant $R>0$ and a smooth graph graph$_u\subset M$ over $\R^n\setminus B_R$ with the graphic function $u=(u^1,\cdots,u^m)$. Moreover, there is a constant $c_M$ such that
\begin{equation}\aligned\label{Decayu}
|D^ju^\a(x)|\le c_M|x|^{-j+1}
\endaligned
\end{equation}
on $\R^n\setminus B_R$ for any $j=0,1,2$ and $1\le \a\le m$. Here, $c_M$ is a general constant, which may change from line to line.

Let $g_{ij}=\de_{ij}+\sum_{1\le\a\le m}u^\a_iu^\a_j$ and $(g^{ij})$ be the inverse matrix of $(g_{ij})$. 
From the equation of self-shrinkers(see \cite{DW} for instance)
\begin{equation}\aligned
\sum_{1\le i,j\le n}g^{ij}u^\a_{ij}=\f{-u^\a+x\cdot Du^\a}2,
\endaligned
\end{equation}
we have
\begin{equation}\aligned
\De_{M}u^\a=&\f1{\sqrt{\mathrm{det}g_{ij}}}\p_{x_i}\left(g^{kl}\sqrt{\mathrm{det}g_{ij}}u^\a_j\right)\\
=&\f1{\sqrt{\mathrm{det}g_{ij}}}\p_{x_i}\left(g^{ij}\sqrt{\mathrm{det}g_{kl}}\right)u^\a_j+\f12x\cdot Du^\a-\f{u^\a}2.
\endaligned
\end{equation}
Denote $g^{ij}_t(x)=g^{ij}(x,t)=g^{ij}\left(\f{x}{\sqrt{t}}\right)$, then
\begin{equation}\aligned
\left|\de_{ij}-g^{ij}_t\right|\le c_1\sum_\be|\na_{\R^n}u^\be|,
\endaligned
\end{equation}
where $c_1$ is a constant.
Let $Q(x,t,Du^\be,D^2u^\g)=\f1{\sqrt{t}}\left(\de_{ij}-g^{ij}_t\right)u^\a_{ij}\big|_{\f{x}{\sqrt{t}}}$, then on $(\R^{n}\setminus B_R)\times\R^+$, from \eqref{Decayu} one has
\begin{equation}\aligned\label{4.27}
|Q(x,t,Du^\be,D^2u^\g)|\le \f{c_2}{|x|}\sum_\be|\na_{\R^n}u^\be|,
\endaligned
\end{equation}
where $c_2$ is a constant.

Denote $a^{ij}(x,t)=a^{ij}_0\left(\f{x}{\sqrt{t}}\right)$ and $U^\a(x,t)=\sqrt{t}u^\a\left(\f{x}{\sqrt{t}}\right)$. Then
\begin{equation}\aligned
&\f{\p}{\p t}U^\a+\De_{\R^n}U^\a=\f1{2\sqrt{t}}u^\a\left(\f{x}{\sqrt{t}}\right)-\f12Du^\a\left(\f{x}{\sqrt{t}}\right)\cdot\f{x}{t}
+\f1{\sqrt{t}}\De_{\R^n}u^\a\bigg|_{\f{x}{\sqrt{t}}}\\
=&-\f1{\sqrt{t}}g^{ij}_tu^\a_{ij}+\f1{\sqrt{t}}\De_{\R^n}u^\a\bigg|_{\f{x}{\sqrt{t}}}=Q(x,t,Du^\be,D^2u^\g).
\endaligned
\end{equation}
Hence for any $(x,t)\in \left(\R^{n}\setminus B_R\right)\times\R^+$, combining \eqref{4.27} we have
\begin{equation}\aligned
\left|\f{\p}{\p t}U^\a+\De_{\R^n}U^\a\right|\le\f{c_2}{|x|}\sum_\be|\na_{\R^n}U^\be|.
\endaligned
\end{equation}
Due to Theorem 1 (with the version of vector-valued functions) showed by Escauriaza-Seregin-$\mathrm{\check{S}}$ver$\mathrm{\acute{a}}$k in \cite{ESS} (see the following content in Theorem 1 of \cite{ESS}), we obtain
$$U^\a\equiv0\qquad \mathrm{on}\ \ \R^{n}\setminus B_R,$$
and then graph$_u$ is an $n$-plane through the origin. Hence $M$ is an $n$-plane through the origin by the rigidity of elliptic equations, and then we complete the proof.
\end{proof}

\section{Appendix}

Let us prove Theorem \ref{epreg}.
There exist $\sigma_1\in(0,1)$, $t_1\in[t_0-(1-\sigma_1)^2,t_0]$ and $X_1\in\mathcal{M}_{t_1}\cap \overline{B}_{1-\sigma_1}(X_0)$ such that
$$\sigma_1^2|B|^2\Big|_{(X_1,t_1)}=\sup_{\sigma\in[0,1]}\left(\sigma^2\sup_{t\in(t_0-(1-\sigma)^2,t_0)}\ \sup_{\mathcal{M}_t\cap B_{1-\sigma}(X_0)}|B|^2\right).$$
Denote $\lambda_1=|B|^{-1}\Big|_{(X_1,t_1)}$. Then
$$\sup_{t\in(t_0-(1-\f{\sigma_1}2)^2,t_0)}\ \sup_{\mathcal{M}_t\cap B_{1-\f{\sigma_1}2}(X_0)}|B|^2\le\f4{\lambda_1^2}.$$
Since
$$B_{\f{\sigma_1}2}(X_1)\times\left(t_1-\f{\sigma_1^2}4,t_1\right)\subset B_{1-\f{\sigma_1}2}(X_0)\times\left(t_0-\left(1-\f{\sigma_1}2\right)^2,t_0\right),$$
then
$$\sup_{t\in(t_1-\f{\sigma_1^2}4,t_1)}\ \sup_{\mathcal{M}_t\cap B_{\f{\sigma_1}2}(X_1)}|B|^2\le\f4{\lambda_1^2}.$$
Let $I_{X_0,t_0}$ be as in \eqref{DefIx0t0}.
It is sufficient to prove
$$\sigma_1\lambda_1^{-1}\le\left(\ep_0^{-1}I_{X_0,t_0}\right)^{\f1p}$$
for a certain uniform constant $\ep_0>0$ depending only on $n$ provided $I_{X_0,t_0}\le\ep_0$. By contradiction, we assume
$$\sigma_1\lambda_1^{-1}>\left(\ep_0^{-1}I_{X_0,t_0}\right)^{\f1p}.$$
Denote $\lambda\triangleq\lambda_1\left(\ep_0^{-1}I_{X_0,t_0}\right)^{\f1p}<\sigma_1$.

Define
$$\widetilde{M}_s=\lambda^{-1}\left(M_{\lambda^2s+t_1}-X_1\right)$$
for $s\in\left(-\f{4+t_1}{\lambda^2},\f{t_0-t_1}{\lambda^2}\right)$, where we have changed variables by setting $X=\lambda Y+X_1$ and $t=\lambda^2s+t_1$. Then $\widetilde{M}_s$ is a smooth solution of mean curvature flow satisfying
$$0\in\widetilde{M}_0,\qquad |B|\Big|_{(0,0)}=\left(\ep_0^{-1}I_{X_0,t_0}\right)^{\f1p}\le1$$
and
$$\sup_{s\in(-\f{\sigma_1^2}{4\lambda^2},0)}\sup_{\widetilde{M}_s\cap B_{\f{\sigma_1}{2\lambda}}}|B|^2\le4\left(\ep_0^{-1}I_{X_0,t_0}\right)^{\f2p}.$$
Since $\sigma_1>\lambda$, then
$$\sup_{s\in(-\f14,0)}\sup_{\widetilde{M}_s\cap B_{\f{1}{2}}}|B|^2\le4\left(\ep_0^{-1}I_{X_0,t_0}\right)^{\f2p}.$$
By scaling, it follows that
\begin{equation}\aligned
I_{X_0,t_0}=\sup_{\sqrt{-t_0}\le\r<\r'\le2}\left(\f{\lambda^2}{\r'^2-\r^2}\right)^{\f{n+2-p}2}\int_{-\f{\r'^2+t_1}{\lambda^2}}^{-\f{\r^2+t_1}{\lambda^2}}
\int_{\widetilde{M}_s\cap B_{\f2{\lambda}}\left(\f{X_0-X_1}{\lambda}\right)}|B|^p.
\endaligned
\end{equation}
Since $-1<t_0<0$ and $t_0-(1-\sigma_1)^2\le t_1\le t_0$, we choose $\r^2=-t_1$, $\r'^2-\r^2=\r'^2+t_1=2\lambda^2>0$. Noting $X_1\in\mathcal{M}_{t_1}\cap \overline{B}_{1-\sigma_1}(X_0)$, so we have
\begin{equation}\aligned
I_{X_0,t_0}\ge2^{-\f{n+2-p}2}\int_{-2}^0\int_{\widetilde{M}_s\cap B_{\f1{\lambda}}(0)}|B|^p\ge2^{-\f{n+2-p}2}\int_{-\f14}^0\int_{\widetilde{M}_s\cap B_{\f12}}|B|^p.
\endaligned
\end{equation}
Now let's recall the evolution equation for the norm of second fundamental form in \cite{X}:
\begin{equation}\aligned
\left(\f{d}{ds}-\De_{\widetilde{M}_s}\right)|B|^2=-2|\na B|^2+2|R^N|+2\sum_{\a,\be}S_{\a\be}^2\le3|B|^4.
\endaligned
\end{equation}
Since
$$\sup_{s\in(-\f14,0)}\sup_{\widetilde{M}_s\cap B_{\f{1}{2}}}|B|^2\le4\left(\ep_0^{-1}I_{X_0,t_0}\right)^{\f2p}\le4,$$
then
\begin{equation}\aligned
\left(\f{d}{ds}-\De_{\widetilde{M}_s}\right)|B|^p\le\f{3p}2|B|^{p+2}\le6p|B|^p.
\endaligned
\end{equation}
By the mean value inequality for mean curvature flow in \cite{E0}\cite{E} (where the case of submanifolds is similar to the case of hypersurfaces), there exists a constant $c(n)$ such that
\begin{equation}\aligned
|B|^p\Big|_{(0,0)}\le c(n)\int_{-\f14}^0\int_{\widetilde{M}_s\cap B_{\f12}}|B|^p,
\endaligned
\end{equation}
which implies
\begin{equation}\aligned
\ep_0^{-1}I_{X_0,t_0}\le c(n)2^{\f{n+2-p}2}I_{X_0,t_0}.
\endaligned
\end{equation}
This is impossible for the sufficiently small $\ep_0$. Hence we complete the proof of Theorem \ref{epreg}.

\bibliographystyle{amsplain}

\end{document}